\documentclass{article}

\usepackage{arxiv}

\usepackage[utf8]{inputenc} 
\usepackage[T1]{fontenc}    
\usepackage{hyperref}       
\usepackage{url}            
\usepackage{booktabs}       
\usepackage{amsmath,amssymb,amsfonts,amsthm}
\usepackage{xcolor}

\newtheorem{theorem}{Theorem}[section]
\newtheorem{lemma}[theorem]{Lemma}
\newtheorem{corollary}[theorem]{Corollary}
\newtheorem{proposition}[theorem]{Proposition}
\newtheorem{property}[theorem]{Property}
\newtheorem{definition}[theorem]{Definition}

\usepackage{nicefrac}       
\usepackage{microtype}      
\usepackage{cleveref}       
\usepackage{lipsum}         
\usepackage{graphicx}
\usepackage{natbib}
\usepackage{doi}

\usepackage[english]{babel}

\title{Necessary and Sufficient Conditions for the Existence of an LU Factorization for General Rank Deficient Matrices\thanks{Code is available at \url{https://doi.org/10.5281/zenodo.18224816}}}
\date{\today}

\author{Eric Darve\\
Institute for Computational and Mathematical Engineering\\	
Mechanical Engineering Department\\
Stanford University\\
Stanford, CA 94305, USA\\
\texttt{darve@stanford.edu}
}

\renewcommand{\headeright}{Technical Report}
\renewcommand{\undertitle}{Technical Report}
\renewcommand{\shorttitle}{Conditions for the Existence of LU Factorizations}	

\usepackage{listings}
\usepackage{csquotes}
\MakeOuterQuote{"}

\lstdefinestyle{PythonStyle}{
    language=Python,
    basicstyle=\ttfamily\fontsize{7pt}{8pt}\selectfont, 
    keywordstyle=\color{blue}\bfseries, 
    stringstyle=\color{red}, 
    commentstyle=\color{green}\itshape, 
    showstringspaces=false, 
    numbers=left, 
    numberstyle=\tiny\color{gray}, 
    breaklines=true, 
    captionpos=b, 
    escapeinside={(*@}{@*)}, 
}

\DeclareMathOperator{\rk}{rank}
\DeclareMathOperator{\nl}{null}
\DeclareMathOperator{\spn}{span}

\newcommand{\by}{\times}

\Crefname{property}{Property}{Properties}
\crefname{property}{property}{properties}
\Crefname{corollary}{Corollary}{Corollaries}
\crefname{corollary}{corollary}{corollaries}
\Crefname{lemma}{Lemma}{Lemmas}
\crefname{lemma}{lemma}{lemmas}
\Crefname{proposition}{Proposition}{Propositions}
\crefname{proposition}{proposition}{propositions}
\Crefname{theorem}{Theorem}{Theorems}
\crefname{theorem}{theorem}{theorems}
\Crefname{lstlisting}{Algorithm}{Algorithms}
\crefname{lstlisting}{algorithm}{algorithms}

\addto\captionsenglish{}
\addto\captionsenglish{}

\begin{document}

\maketitle

\begin{abstract}
	We establish necessary and sufficient conditions for the existence of an LU factorization $A=LU$ for an arbitrary square matrix $A$, including singular and rank-deficient cases, without the use of row or column permutations. We prove that such a factorization exists if and only if the nullity of every leading principal submatrix is bounded by the sum of the nullities of the corresponding leading column and row blocks. While building upon the work of Okunev and Johnson, we present simpler, constructive proofs. Furthermore, we extend these results to characterize rank-revealing factorizations, providing explicit sparsity bounds for the factors $L$ and $U$. Finally, we derive analogous necessary and sufficient conditions for the existence of factorizations constrained to have unit lower or unit upper triangular factors.
\end{abstract}

\keywords{LU factorization, Rank-deficient matrices, Singular matrices, Rank-revealing factorization, Nullity, Sparsity}

It is well-known that for any square matrix $A$, there exists a permutation of rows (and potentially columns) such that an LU factorization exists. Specifically, the factorization $PA = LU$, where $P$ is a permutation matrix, $L$ is unit lower triangular, and $U$ is upper triangular, is guaranteed to exist for any matrix $A$~\cite{higham_gaussian_2011,lu_numerical_2021}.

However, the existence of an LU factorization for a matrix $A$ \emph{without} permutations (that is, $A = LU$) is not guaranteed. A classical result states that for a non-singular matrix, such a factorization exists if and only if all its leading principal minors are non-zero. When the matrix is singular or when leading principal minors vanish, the situation becomes significantly more complex. The standard conditions involving determinants are no longer sufficient to characterize the existence of the factorization.

This paper addresses the necessary and sufficient conditions for the existence of an LU factorization for an arbitrary matrix $A$, without imposing restrictions on its rank or invertibility. We extend the work of Okunev and Johnson~\cite{okunev_necessary_2005}, who established rank-based conditions for this problem. See also~\cite{dopico_multiple_2006} for a discussion on multiple factorizations of singular matrices. Other perspectives on factorizations and their existence in general algebraic structures or specific contexts can be found in~\cite{nagarajan_products_1999,johnson_inherited_1989,jeffrey_lu_2010,higham_gaussian_2011}.

The contributions of this paper are threefold. First, we formulate the necessary and sufficient condition using the \emph{nullity} of the leading principal submatrices relative to the nullity of the corresponding column and row blocks. Specifically, we prove that an LU factorization exists if and only if for all $k$ (this condition is equivalent to~\cite{okunev_necessary_2005}):
\[
	\nl(A[1:k,1:k]) \le \nl(A[:,1:k]) + \nl({A[1:k,:]}^T)
\]
Note that, in this paper, null is the dimension of the null space of a matrix (see~\Cref{tab:notations}).

Second, we provide new, simpler proofs of these results. Our proofs are constructive and rely on induction, explicitly building the factorization by handling the singular cases while preserving the triangular structure. Finally, we extend these results to characterize rank-revealing factorizations, providing tight sparsity bounds for the factors. We also derive analogous conditions for factorizations requiring unit lower or unit upper triangular factors.

\begin{table}[htbp]
	\centering
	\begin{tabular}{ll}
		\toprule
		$\rk(A)$     & Rank of matrix $A$ \\
		$\nl(A)$     & Dimension of the right null space of matrix $A$ \\
		$A^T$        & Transpose of matrix $A$ \\
		$a_{ij}$     & Entry $(i,j)$ of matrix $A$ \\
		$A[:,i]$     & Column $i$ of matrix $A$ \\
		$A[i,:]$     & Row $i$ of matrix $A$ \\
		$A[i:j,k:l]$ & Submatrix of $A$ with rows from $i$ to $j$ and columns from $k$ to $l$ \\
		$A[i:j,k:]$  & Submatrix of $A$ with rows from $i$ to $j$ and columns from $k$ to the end \\
		$e_i$        & $i$-th vector of the standard basis \\
		$I_n$        & Identity matrix of size $n$ \\
		\bottomrule
	\end{tabular}
	\vspace{5pt}
	\caption{\label{tab:notations} Notations used in the paper.}
\end{table}

\paragraph{Triangular factorizations and causal systems.} Before presenting the technical results, we motivate the study of LU factorizations without permutations by considering their interpretation in causal systems and their computational advantages.

In some applications, the index ordering carries physical significance, such as time in dynamic systems. A lower triangular matrix $L$ represents a causal operator where the current state $i$ depends only on past states $j \le i$. Similarly, $U$ can be interpreted as capturing reverse-flow dependencies. When $A=LU$, the system respects this intrinsic ordering. However, introducing permutations (as in $PAQ = LU$) destroys this structure: $P$ scrambles the equations (rows) and $Q$ scrambles the variables (columns), breaking the link between the matrix indices and the causal dependencies. Consequently, in contexts where the ordering is immutable (e.g., autoregressive time series, structural equation modeling, or causal signal processing), standard pivoting strategies are inapplicable. This necessitates a theory for the existence of LU factorizations that strictly preserves the original indexing.

Moreover, the choice of the LU factorization over other decompositions like QR or SVD is often driven by computational efficiency and sparsity. While the QR decomposition always exists for any matrix, the orthogonal factor $Q$ is typically dense. In contrast, the triangular factors $L$ and $U$ often preserve the sparsity of $A$ to a much greater extent, leading to significant savings in storage and computational cost. This cost advantage may also extend to physical limitations, such as data transport or logistical constraints when they are represented by the matrix entries. These reasons lead us to the question of determining exactly when an LU factorization exists without permutations.

We finally summarize the contributions of this paper, which include:
\begin{itemize}
	\item A Nullity-Based Characterization: We prove that factorization exists if and only if $\text{null}(A_{11}) \le \text{null}(\text{Col}) + \text{null}(\text{Row})$ for all $k$. The proofs are original.
	\item Rank-Revealing Structure: we show that we obtain a rank-revealing factorization when the LU factorization exists.
	\item Sparsity: we provide tight sparsity bounds for the L and U factors.
	\item Constructive Proof: we provide a constructive proof by induction. We include a Python implementation in the appendix to fully characterize the algorithm and allow independent verification of our results.
	\item Unit Triangular Factors: we provide a characterization of the conditions for the existence of unit triangular factors.
\end{itemize}

\paragraph{Organization of paper.}
The remainder of the paper is organized as follows. \Cref{sec:classical} reviews classical results regarding factorizations with permutations. \Cref{sec:main} presents the main theorem and the derivation of the nullity condition. \Cref{thm:main2} strengthens this result by characterizing rank-revealing factorizations. Finally, \Cref{sec:unit_triangular} discusses the specific requirements for unit triangular factors.

We begin with a few classical results that will serve as reference for the later sections.

\section{Classical results}\label{sec:classical}

We begin by recalling standard existence results involving permutations. These theorems correspond to Gaussian elimination with partial and full pivoting.

\begin{theorem}[Partial Pivoting]\label{thm:row_permutation}
	For any square matrix $A$ of size $n$, there exists a permutation matrix $P$ such that $PA = LU$, where $L$ is unit lower triangular and $U$ is upper triangular.
\end{theorem}

\begin{proof}
	This is a classical result. We can prove it by induction on $n$. The result is true for $n=1$.

	Assume it is true for $n-1$. If $a_{11} \neq 0$, then consider:
	\[
		A = LDU
	\]
	where
	\[
		L = \begin{pmatrix}
			1                  & 0 \\
			a_{11}^{-1} A_{21} & I
		\end{pmatrix}, \quad
		D = \begin{pmatrix}
			1 & 0 \\
			0 & S
		\end{pmatrix}, \quad
		U = \begin{pmatrix}
			a_{11} & A_{12} \\
			0      & I
		\end{pmatrix}
	\]
	We can apply the induction hypothesis to $S$. The result for $A$ follows.

	If $a_{11} = 0$, we have two cases.

	Case 1: column 1 of $A$ is non-zero. We can find the smallest $i$ such that $A[i,1] \neq 0$. Pivot row 1 and $i$. Denote by $B = \Pi_i A$ the permuted matrix (with $\Pi_i^2 = I$). Since $b_{11} \neq 0$, $B$ has an LU factorization. This proves the result for $A$.

	Case 2: column 1 of $A$ is zero. We choose:
	\[
		A = I D U
	\]
	where
	\[
		D = \begin{pmatrix}
			1 & 0 \\
			0 & S
		\end{pmatrix}, \quad
		U = \begin{pmatrix}
			0 & A_{12} \\
			0 & I
		\end{pmatrix}
	\]
	and $S=A_{22}$. We can apply the induction hypothesis to $S$. The result for $A$ follows.
\end{proof}

\begin{corollary}\label{cor:col_permutation}
	For any square matrix $A$ of size $n$, there exists a permutation matrix $Q$ such that $AQ = LU$, where $L$ is lower triangular and $U$ is unit upper triangular.
\end{corollary}

\begin{proof}
	Apply \Cref{thm:row_permutation} to $A^T$.
\end{proof}

\begin{definition}\label{def:rank_revealing}
	A matrix $A$ of size $m \by n$ is said to have a \emph{rank revealing LU factorization} if it can be written as $A = LU$, where $L$ is a lower triangular matrix of size $m \by r$, $U$ is an upper triangular matrix of size $r \by n$, and $r = \rk(A)$.
\end{definition}

If we allow both row and column permutations, we can always obtain such a factorization.

\begin{theorem}[Full Pivoting]\label{thm:row_column_permutation}
	For any matrix $A$ of size $m \by n$, there exist a row permutation matrix $P$ and a column permutation matrix $Q$ such that $PAQ = LU$ is a rank revealing LU factorization.
\end{theorem}

\begin{proof}
	This is also a classical result. We can prove it by induction on $n$. The result is true for $n=1$.

	Assume it is true for $n-1$. If $a_{11} \neq 0$, then consider:
	\[
		A = LDU
	\]
	where
	\[
		L = \begin{pmatrix}
			a_{11} & 0 \\
			A_{21} & I
		\end{pmatrix}, \quad
		D = \begin{pmatrix}
			a_{11}^{-1} & 0 \\
			0           & S
		\end{pmatrix}, \quad
		U = \begin{pmatrix}
			a_{11} & A_{12} \\
			0      & I
		\end{pmatrix}
	\]
	$S$ has rank $r-1$ if $r = \rk(A)$. We can apply the induction hypothesis to $S$. The result for $A$ then follows.

	If $a_{11} = 0$, we have two cases. If $A$ is 0, we pick $L=U=0$. Otherwise we have some $i$ and $j$ such that $a_{ij} \neq 0$. We can permute row 1 and $i$ and column 1 and $j$. Denote by $B = \Pi_i A \Pi_j$ the permuted matrix. Since $b_{11} \neq 0$, $B$ has a rank revealing LU factorization. This proves the result for $A$.

	We have also proved that this process must terminate after $r$ steps where $r = \rk(A)$. So the LU factorization of $B$ is rank revealing.

	From \Cref{lem:oblique} (see \Cref{sec:appendix} in the Appendix), we can also prove that the remaining Schur complement is zero after $r$ steps of the LU factorization.
\end{proof}

We now consider LU factorizations without permutations. The fundamental result states that factorization is possible if and only if the matrix is \emph{strongly non-singular} (i.e., all leading principal minors are non-zero).

\begin{proposition}\label{prop:sufficient}
	Let $A$ be a non-singular matrix of size $n$. $A$ has an LU factorization if and only if all its leading principal submatrices are non-singular:
	\[
		\rk(A[1:k,1:k]) = k, \quad k = 1, \ldots, n
	\]
\end{proposition}

\begin{proof}
	We prove that the condition is necessary. Assume $A = LU$:
	\[ \det(A) = \prod_{i=1}^n l_{ii} u_{ii} \]
	and $\det(A) \neq 0$. So, $l_{ii} \neq 0$ and $u_{ii} \neq 0$ for all $i$. Therefore: $\rk(A[1:k,1:k]) = k$, $k = 1, \ldots, n$.

	We prove that the condition is sufficient. We prove the result by induction on $n$. The result is true for $n=1$.

	Assume it is true for $n-1$. By assumption, $a_{11} \neq 0$. We then have the factorization:
	\[
		A = LDU
	\]
	where
	\[
		L = \begin{pmatrix}
			a_{11} & 0 \\
			A_{21} & I
		\end{pmatrix}, \quad
		D = \begin{pmatrix}
			a_{11}^{-1} & 0 \\
			0           & S
		\end{pmatrix}, \quad
		U = \begin{pmatrix}
			a_{11} & A_{12} \\
			0      & I
		\end{pmatrix}
	\]
	and $S = A_{22} - A_{21} a_{11}^{-1} A_{12}$ is the Schur complement. $L$ and $U$ are full rank. So $D$ satisfies the condition of the theorem. Consequently, $S$ also satisfies the condition of the theorem and is of size $n-1$. By the induction hypothesis, $S$ has an LU factorization. Therefore $A$ also has an LU factorization.
\end{proof}

\section{General necessary and sufficient condition for the existence of an LU factorization}\label{sec:main}

We now address the central question of this paper: determining the necessary and sufficient conditions for the existence of an LU factorization without permutations. We begin with two illustrative examples that motivate the general result.

Consider the matrix:
\[
	A = \begin{pmatrix} 0 & 1 \\ 1 & 0 \end{pmatrix}
\]
This matrix does not admit an LU factorization. We prove this by contradiction. Suppose $A = LU$. Then the diagonal entry satisfies $a_{11} = l_{11} u_{11} = 0$. This implies that either $l_{11} = 0$ or $u_{11} = 0$.
\begin{itemize}
	\item If $l_{11} = 0$, then $a_{12} = l_{11} u_{12} = 0$, which contradicts $a_{12} = 1$.
	\item If $u_{11} = 0$, then $a_{21} = l_{21} u_{11} = 0$, which contradicts $a_{21} = 1$.
\end{itemize}

In this special case, we have $A^2 = I$ and $A$ is a permutation matrix. So we note that $A A = I \cdot I$ is the factorization of \Cref{thm:row_permutation} with $P = A$, $L = I$, and $U = I$.

However, a zero pivot does not always preclude factorization. Consider the singular matrix:
\[
	A = \begin{pmatrix} 0 & 0 & 0 \\ 0 & 0 & 1 \\ 0 & 1 & 0 \end{pmatrix}
\]
Despite the zero principal minors, this matrix \emph{does} have an LU factorization. We can derive it by permuting $A$ into a factorizable form and then reversing the permutations. Apply a cyclic row shift $(1,2,3) \to (2,3,1)$ via $P$ and swap columns 1 and 3 via $Q$:
\[
	PAQ = \begin{pmatrix} 1 & 0 & 0 \\ 0 & 1 & 0 \\ 0 & 0 & 0 \end{pmatrix}
	= \begin{pmatrix} 1 & 0 \\ 0 & 1 \\ 0 & 0 \end{pmatrix}
	\begin{pmatrix} 1 & 0 & 0 \\ 0 & 1 & 0 \end{pmatrix}
\]
We can recover a factorization for $A$ by writing $A = P^{-1} (PAQ) Q^{-1}$:
\[
	A = P^{-1} \left[ \begin{pmatrix} 1 & 0 \\ 0 & 1 \\ 0 & 0 \end{pmatrix}
		\begin{pmatrix} 1 & 0 & 0 \\ 0 & 1 & 0 \end{pmatrix} \right] Q^{-1} =
	\begin{pmatrix} 0 & 0 \\ 1 & 0 \\ 0 & 1 \end{pmatrix}
	\begin{pmatrix} 0 & 0 & 1 \\ 0 & 1 & 0 \end{pmatrix}
\]
This yields a valid, rank-revealing LU factorization with no pivoting. Another valid LU factorization is:
\[
	A = \begin{pmatrix} 0 & 0 \\ 0 & 1 \\ 1 & 0 \end{pmatrix}
	\begin{pmatrix} 0 & 1 & 0 \\ 0 & 0 & 1 \end{pmatrix}
\]

We can generalize this result using block matrices. Let $0_n$ denote the $n \times n$ zero matrix. If $B$ is a factorizable matrix such that $B=LU$, we can construct larger factorizable matrices $A$ as follows:
\[
	A = \begin{pmatrix}
		0_n & 0_n & 0_n \\
		0_n & 0_n & C   \\
		0_n & B   & D
	\end{pmatrix}
	=
	\begin{pmatrix}
		0_n & 0_n \\
		C   & 0_n \\
		D   & L
	\end{pmatrix}
	\begin{pmatrix}
		0_n & 0_n & I_n \\
		0_n & U   & 0_n
	\end{pmatrix}
	=
	\begin{pmatrix}
		0_n & 0_n \\
		CX  & 0_n \\
		DX  & L
	\end{pmatrix}
	\begin{pmatrix}
		0_n & 0_n & X^{-1} \\
		0_n & U   & 0_n
	\end{pmatrix}
\]
for any invertible matrix $X$.

Alternatively, if $C = LU$:
\[
	A =
	\begin{pmatrix}
		0_n & 0_n \\
		0_n & L   \\
		I_n & 0_n
	\end{pmatrix}
	\begin{pmatrix}
		0_n & B   & D \\
		0_n & 0_n & U
	\end{pmatrix}
	=
	\begin{pmatrix}
		0_n & 0_n \\
		0_n & L   \\
		X   & 0_n
	\end{pmatrix}
	\begin{pmatrix}
		0_n & X^{-1} B & X^{-1} D \\
		0_n & 0_n      & U
	\end{pmatrix}
\]

These examples demonstrate that the existence of a factorization is linked to the relationship between the ranks (or nullities) of specific sub-blocks. We formalize this observation in \Cref{prop:rank}. We will show that \Cref{prop:rank} below is a necessary and sufficient condition for the existence of an LU factorization.

\begin{property}\label{prop:rank}
	Matrix $A$ is square of size $n$. For all $k = 1, \ldots, n$:
	\begin{equation}
		\label{eq:cond1}
		\nl(A[1:k,1:k]) \le \nl(A[:,1:k]) + \nl({A[1:k,:]}^T)
	\end{equation}
\end{property}

We informally describe the main strategy before stating the results formally. The primary obstacle to constructing an LU factorization is the occurrence of a zero pivot. Standard Gaussian elimination resolves this by permuting rows (partial pivoting) or both rows and columns (full pivoting) to bring a non-zero entry to the diagonal.

However, arbitrary permutations destroy the triangular structure required for the specific factorization $A=LU$. If we write the permuted factorization as $P A Q = L U$, recovering $A$ yields $A = (P^{-1} L) (U Q^{-1})$. In general, $P^{-1} L$ is not lower triangular and $U Q^{-1}$ is not upper triangular.

To maintain the triangular factors of $A$, we must restrict our pivoting strategy. We rely on the observation (detailed in \Cref{sec:appendix}) that a zero column (or row) in the Schur complement corresponds to a column (or row) in the original matrix that is linearly dependent on the preceding columns (or rows).

Our constructive proof effectively "sorts" the matrix indices. Whenever we encounter a zero pivot, the condition in \Cref{prop:rank} guarantees that at least one of the following two scenarios must hold:
\begin{enumerate}
	\item The current column in the Schur complement is zero (the column is linearly dependent). We permute this column rightward.
	\item The current row in the Schur complement is zero (the row is linearly dependent). We permute this row downward.
\end{enumerate}

Crucially, because we only permute \emph{zero} columns or rows (relative to the active Schur complement), the triangular structure of the resulting factors is preserved. When we reverse these permutations, the factor $P^{-1} L$ remains lower triangular and $U Q^{-1}$ remains upper triangular.

Condition~\eqref{eq:cond1} is the precise safeguard that ensures we never reach a "dead end" where the pivot is zero but neither the column nor the row can be safely permuted away (i.e., a case where the rank has not been exhausted, but the leading block is rank-deficient in a way that blocks factorization).

Consequently, this process produces a rank-revealing factorization where the indices are partitioned based on linear independence:
\begin{itemize}
	\item Linearly Independent Vectors: If a column $j_0$ of $A$ is linearly independent of the previous columns, it is retained (or moved) into the leading rank-$r$ block ($j \le r$).
	\item Linearly Dependent Vectors: If a column $j_0$ of $A$ is linearly dependent on the previous columns, it is deferred to the trailing block ($j > r$).
\end{itemize}
A symmetric logic applies to the rows of $A$.

The main existence results are summarized in \Cref{tab:main}.

\begin{table}[htbp]
	\centering
	\begin{tabular}{llp{4cm}p{6cm}}
		\toprule
		Factorization Type & Equation & Conditions for Existence & Limitations \\
		\midrule
		Non-singular LU & $A=LU$ & Strongly non-singular (all leading minors $\neq 0$). & Does not apply to singular matrices. \\
		Partial Pivoting & $PA=LU$ & \raggedright Always exists for any square $A$. & \raggedright Permutes rows, altering row-index meaning.\tabularnewline
		Full Pivoting & $PAQ=LU$ & \raggedright Always exists; rank-revealing. & \raggedright Permutes rows and columns, destroying all index structure.\tabularnewline
		Generalized LU & $A=LU$ & Subject of this work. & \raggedright Handles singular/rank-deficient cases without $P$, $Q$.\tabularnewline
		\bottomrule
	\end{tabular}
	\vspace{5pt}
	\caption{Summary of the main factorization types and their conditions for existence.}\label{tab:main}
\end{table}

\subsubsection*{Informal description: producer-sensor duality}

We can form a mental model of \Cref{prop:rank} by considering the following scenario which represents the class of matrices that have an LU factorization. Assume the columns of $A$ correspond to producer points $j$ and the rows of $A$ correspond to sensor points $i$.

We say the producer $j$ is \emph{observable} by the current sensor network $1:i$ if $A[1:i,j]$ adds new information (is linearly independent) relative to previous producers $A[1:i,1:j-1]$. Conversely, we say the sensor $i$ is \emph{informative} regarding the current producers $1:j$ if its measurement $A[i,1:j]$ is linearly independent of previous measurements $A[1:i-1,1:j]$.

The existence of an LU factorization implies a compatibility between observability and informativeness. $A$ admits a factorization if and only if every "blind spot" (rank deficiency) in a leading block corresponds to a global redundancy. Specifically, if the interaction between the first $k$ sensors and $k$ producers is rank-deficient, this deficiency must be fully accounted for by producers that are globally redundant (dependent on previous columns) or sensors that are globally redundant (dependent on previous rows). The factorization is impossible only in the "crossed" scenario, or "deadlock": where producer $k$ carries a novel signal that current sensors miss (but future sensors see), while sensor $k$ performs a novel measurement that misses current signals (but sees future ones).

We now provide the formal proof of the main results, \Cref{thm:main,thm:main2}.

\subsection{Rank theorems}

We recall rank theorems that will be useful in the proof of \Cref{thm:main}.

\begin{theorem}\label{thm:rank-nullity}
	The rank-nullity theorem states that for any matrix $A$ of size $m \times n$:
	\[
		\rk(A) + \nl(A) = n
	\]
\end{theorem}

We now state Sylvester's inequality for the rank of a product of matrices.

\begin{lemma}\label{lem:rank}
	Consider $A = BC$. Then:
	\[
		\rk(A) \le \min(\rk(B), \rk(C)), \quad
		\rk(B) + \rk(C) - k \le \rk(A)
	\]
	where $k$ is the number of columns of $B$ and rows of $C$.
\end{lemma}

\begin{corollary}\label{cor:nullity}
	If $A$, $B$, and $C$ are square matrices then:
	\[
		\nl(A) \le \nl(B) + \nl(C)
	\]
\end{corollary}

\begin{proof}
	Assume $A = BC$ where $A$, $B$, $C$ are $n \times n$ matrices. From \Cref{lem:rank}:
	\[ \rk(A) \ge \rk(B) + \rk(C) - n \]
	Using \Cref{thm:rank-nullity}, $\rk(M) = n - \nl(M)$:
	\[ n - \nl(A) \ge (n - \nl(B)) + (n - \nl(C)) - n \]
	Simplifying: $\nl(A) \le \nl(B) + \nl(C)$.
\end{proof}

\subsection{Necessary condition}

We show that \Cref{prop:rank} is a necessary condition for the existence of an LU factorization.

Assume that $A$, a square matrix of size $n$, has an LU factorization:
\[ A = LU \]
Then:
\begin{align*}
	A[1:k,:]   & = L[1:k,1:k] U[1:k,:], \\
	A[:,1:k]   & = L[:,1:k] U[1:k,1:k], \\
	A[1:k,1:k] & = L[1:k,1:k] U[1:k,1:k]
\end{align*}
Using \Cref{lem:rank} and \Cref{cor:nullity}, we have:
\begin{align*}
	\rk(A[1:k,:])   & \le \rk(L[1:k,1:k]), \\
	\rk(A[:,1:k])   & \le \rk(U[1:k,1:k]), \\
	\nl(A[1:k,1:k]) & \le \nl(L[1:k,1:k]) + \nl(U[1:k,1:k])
\end{align*}
Using \Cref{thm:rank-nullity}, we get:
\[
	\nl(A[1:k,1:k]) \le \nl(A[:,1:k]) + \nl({A[1:k,:]}^T)
\]
So \Cref{prop:rank} is a necessary condition for the existence of an LU factorization.

\subsection{Sufficient condition}

We will now prove the main theorem:
\begin{theorem}\label{thm:main}
	Matrix $A$ satisfies \Cref{prop:rank} if and only if it has an LU factorization.
\end{theorem}

\begin{proof}
	We prove the result by induction on the size of $A$. The result is true for $n=1$.

	Assume it is true for size $n-1$.

	Assume $a_{11} \neq 0$. We have the factorization:
	\[
		A = LDU
	\]
	where
	\[
		L = \begin{pmatrix}
			a_{11} & 0 \\
			A_{21} & I
		\end{pmatrix}, \quad
		D = \begin{pmatrix}
			a_{11}^{-1} & 0 \\
			0           & S
		\end{pmatrix}, \quad
		U = \begin{pmatrix}
			a_{11} & A_{12} \\
			0      & I
		\end{pmatrix}
	\]
	Since $L$ and $U$ are triangular and full rank, $D$ satisfies \Cref{prop:rank} if and only if $A$ does. Consequently, $S$ also satisfies \Cref{prop:rank} and is of size $n-1$. Moreover $S$ has rank $r-1$ if $r = \rk(A)$. By the induction hypothesis, $S$ has an LU factorization. The result for $A$ follows.

	Assume $a_{11} = 0$. From \Cref{prop:rank}, we have either:
	\begin{enumerate}
		\item $A[:,1] = 0$, $A[1,:] \ne 0$, or
		\item $A[1,:] = 0$, $A[:,1] \ne 0$, or
		\item $A[:,1] = 0$, $A[1,:] = 0$.
	\end{enumerate}

	Case 1: $A[:,1] = 0$. We can find the smallest $j$ such that $A[1,j] \neq 0$. Pivot column 1 and $j$. Denote by $B = A \Pi_j$ the permuted matrix (with $\Pi_j^2 = I$). We prove that $B$ satisfies \Cref{prop:rank}.

	Consider $k < j$. We have
	\begin{equation} \label{eq:1}
		\nl(A[1:k,1:k]) \le \nl(A[:,1:k]) + \nl({A[1:k,:]}^T)
	\end{equation}
	Since column 1 of $A$ is zero, we have:
	\[
		\nl(A[1:k,1:k]) = \nl(A[1:k,2:k]) + 1
	\]
	We have
	\[
		\nl(B[1:k,2:k]) = \nl(A[1:k,2:k])
	\]
	Since the only non-zero entry in row 1 of $B[1:k,1:k]$ is in column 1, $B[1:k,1]$ is linearly independent of $B[1:k,2:k]$. Therefore, we have
	\[
		\nl(B[1:k,1:k]) = \nl(B[1:k,2:k]) = \nl(A[1:k,2:k]) = \nl(A[1:k,1:k]) - 1
	\]
	By a similar reasoning, we have
	\[
		\nl(B[:,1:k]) = \nl(A[:,1:k]) - 1
	\]
	The column pivoting does not affect the rank of the rows, so
	\[
		\nl({B[1:k,:]}^T) = \nl({A[1:k,:]}^T)
	\]
	Therefore $B$ satisfies \Cref{prop:rank} for $k < j$.

	For $k \ge j$, the pivoting does not affect the ranks in \Cref{eq:1}, so $B$ also satisfies \Cref{prop:rank} for $k \ge j$.

	So $B$ satisfies \Cref{prop:rank}.

	Since $b_{11} \neq 0$, $B$ has an LU factorization: $B = LU$. Column $j$ of $B$ is zero. If column $j$ of $U$ is non-zero, we can zero it out without affecting any column in $LU$. So we can always choose $U$ such that column $j$ is zero. We have:
	\[ A = B \Pi_j = L U \Pi_j \]
	and $U \Pi_j$ is still upper triangular because $\Pi_j$ only permutes column 1 and $j$. Therefore $A$ has an LU factorization.

	Case 2: $A[1,:] = 0$. Consider $B = A^T$. $B$ corresponds to Case 1, so $B$ has an LU factorization: $B = L U$. Therefore:
	\[ A = B^T = U^T L^T \]
	is an LU factorization of $A$.

	Case 3: If $A$ is 0, it has an LU factorization with $L=U=0$. Otherwise, if $A \neq 0$, let's denote by $j$ the smallest index such that column $j$ is non zero. Let's denote by $i > 1$ the smallest index such that $a_{ij} \neq 0$. Let's pivot column 1 and $j$ and denote by $B = A \Pi_j$ the permuted matrix. $B$ satisfies Case 2, so it has an LU factorization: $B = L U$. Column $j$ of $U$ is zero. Therefore:
	\[ A = B \Pi_j = L U \Pi_j \]
	is an LU factorization of $A$.

	We have also proved that this process must terminate after $r$ steps where $r = \rk(A)$. So the LU factorization of $B$ is rank revealing.
\end{proof}

We now prove a stronger theorem. Please compare this result with \Cref{thm:row_column_permutation}.

\begin{theorem}\label{thm:main2}
	Let $A$ be a matrix of size $n$ with rank $r \le n$. If $A$ satisfies \Cref{prop:rank}, then there exist a row permutation $P_r$ and a column permutation $P_c$ such that the matrix
	\[ B = P_r^{-1} A P_c^{-1} \]
	has strongly non-singular leading principal submatrices $B[1:k,1:k]$ for $1 \le k \le r$. Consequently, the leading principal block admits an LU factorization $B[1:r,1:r] = L_r U_r$.

	Define the rank-revealing factors of $B$ as:
	\[
		L_B = \begin{pmatrix} L_r \\ B[r+1:n, 1:r] U_r^{-1} \end{pmatrix}
		\quad \text{and} \quad
		U_B = \begin{pmatrix} U_r & L_r^{-1} B[1:r, r+1:n] \end{pmatrix}
	\]
	Then $B = L_B U_B$. Furthermore, $A$ admits a rank-revealing factorization
	\[ A = L U \]
	where $L = P_r L_B$ is lower triangular and $U = U_B P_c$ is upper triangular.

	Moreover, we have the following sparsity results for $L$. Let $i_0$ be the row index in $A$ corresponding to the logical row index $i$ in $B$ (i.e., $e_{i_0} = P_r e_i$). Then:
	\begin{itemize}
		\item If $i \le r$, then $i_0 \ge i$ and $L[i_0, i+1:] = 0$. Row $i_0$ of $A$ is linearly independent of the previous rows.
		\item If $i > r$, then $i_0 \le i$; if $i_0 \le r$, then $L[i_0,i_0:] = 0$. Row $i_0$ of $A$ is a linear combination of the previous rows.
	\end{itemize}
	Similarly for $U$, let $j_0$ be the column index in $A$ corresponding to the logical column index $j$ in $B$ (i.e., $e_{j_0} = P_c^T e_j$). Then:
	\begin{itemize}
		\item If $j \le r$, then $j_0 \ge j$ and $U[j+1:, j_0] = 0$. Column $j_0$ of $A$ is linearly independent of the previous columns.
		\item If $j > r$, then $j_0 \le j$; if $j_0 \le r$, then $U[j_0:, j_0] = 0$. Column $j_0$ of $A$ is a linear combination of the previous columns.
	\end{itemize}
\end{theorem}

\begin{proof}
	We follow a proof similar to the one for \Cref{thm:main}.

	We do an induction on the size of $A$. The result is true for $n=1$.

	Assume it is true for size $n-1$. If $a_{11} \neq 0$, we have the factorization:
	\[
		A = LDU
	\]
	where
	\[
		L = \begin{pmatrix}
			a_{11} & 0 \\
			A_{21} & I
		\end{pmatrix}, \quad
		D = \begin{pmatrix}
			a_{11}^{-1} & 0 \\
			0           & S
		\end{pmatrix}, \quad
		U = \begin{pmatrix}
			a_{11} & A_{12} \\
			0      & I
		\end{pmatrix}
	\]
	We apply the induction hypothesis to $S$. We can define $P_r$ and $P_c$ in terms of the permutations for $S$:
	\[ P_r = \begin{pmatrix}
			1 & 0       \\
			0 & P_{r,S}
		\end{pmatrix}, \quad
		P_c = \begin{pmatrix}
			1 & 0       \\
			0 & P_{c,S}
		\end{pmatrix}
	\]
	The sparsity results follow directly from the induction hypothesis.

	Assume $a_{11} = 0$. From \Cref{prop:rank}, we have:
	\begin{enumerate}
		\item $A[:,1] = 0$, $A[1,:] \ne 0$, or
		\item $A[1,:] = 0$, $A[:,1] \ne 0$, or
		\item $A[:,1] = 0$, $A[1,:] = 0$.
	\end{enumerate}

	Case 1: $A[:,1] = 0$. As for \Cref{thm:main}, we can find the smallest $j$ such that $A[1,j] \neq 0$. We pivot column 1 and $j$.

	We need to prove that column $j$ cannot be a linear combination of the previous columns. Assume it is. Then there exists $x$ such that:
	\[ A[:,j] = A[:,1:j-1] x \]
	So:
	\[ A[1,j] = A[1,1:j-1] x = 0 \]
	Since $A[1,j] \neq 0$, we have a contradiction. So column $j$ is linearly independent of the previous columns.

	Denote by $B = A \Pi_j$ the permuted matrix. $B$ satisfies \Cref{prop:rank} and $b_{11} \neq 0$. So $B$ has an LU factorization with the sparsity pattern of \Cref{thm:main2}. Since column $j$ of $B$ is zero, column $j$ of $U_{B}$ is zero. We define
	\[ P_c = P_{c,B} \Pi_j, \quad U = U_B \Pi_j \]
	$\Pi_j$ swaps columns 1 and $j$ of $U_{B}$. We have $e_1 = \Pi_j^T e_j$ and column $j$ of $U_{B}$ is zero. We also have a $j_0$ such that $e_j = P_{c,B}^T e_{j_0}$ and $e_1 = P_c^T e_{j_0}$. Note that since column $j$ of $B$ (which is column 1 of $A$) is zero, we are beyond the rank-revealing part and $j_0 > r$. Since column 1 of $U$ is zero and the other columns of $U$ have the sparsity pattern of \Cref{thm:main2}, $U$ also satisfies the sparsity pattern of \Cref{thm:main2}.

	Case 2: $A[1,:] = 0$. Consider $B = A^T$. $B$ corresponds to Case 1 and satisfies the sparsity pattern of \Cref{thm:main2}. So $A$ also satisfies the sparsity pattern of \Cref{thm:main2}.

	Case 3: If $A = 0$, it has an LU factorization. In that case we pick $L=U=0$ as the factors. If $A \neq 0$, let's denote by $j$ the smallest index such that column $j$ is non zero. Let's denote by $i > 1$ the smallest index such that $a_{ij} \neq 0$. Let's pivot column 1 and $j$ and denote by $B = A \Pi_j$ the permuted matrix. $B$ satisfies Case 2, so it has an LU factorization: $B = L U$ with the sparsity pattern of \Cref{thm:main2}. Following a reasoning analogous to Case 1, we can show that $A$ also satisfies the sparsity pattern of \Cref{thm:main2}.

	As in \Cref{thm:main}, we have also proved that this process must terminate after $r$ steps where $r = \rk(A)$. So the LU factorization of $B$ is rank revealing.

	Following \Cref{lem:oblique}, any column $j_0$ in $A$ that is a linear combination of the previous columns will become 0 at step $\min(r,j_0-1)$. If $j_0 \le r$, this column will be moved rightward to position $j$ with $j > r$. Any column $j_0$ that is a linearly independent of the previous columns will never become zero. If $j_0 > r$, it will be moved leftward to position $j$ with $j \le r$.

	A similar result applies to the rows of $A$.
\end{proof}

\subsection{Python implementation}

In \Cref{sec:python_restricted_pivoting}, we provide a Python implementation of the algorithm. The algorithm returns L and U such that \verb|A = L @ U|.

Please compare against the algorithm with full pivoting in \Cref{sec:python_full_pivoting}, which returns P, Q, L, and U such that \verb|P @ A @ Q = L @ U|. Full pivoting selects the largest pivot for maximum stability. Our algorithm selects the pivot to ensure that the factors L and U remain triangular.

\Cref{lst:lu_restricted_pivoting} is vectorized for performance. The internal permutations are done implicitly. For example, line~\ref{copy_leftward} moves column $j$ to column $k$. Line~\ref{zero_column} zeros out column $j$ of $A$ in exact arithmetic.

This algorithm is not backward stable as we do not control the size of the pivot. Finding a backward stable algorithm will be the topic of future work.

The pseudo-code in~\cite{okunev_necessary_2005} (Algorithm 1, p.~12) does not consider the case where the factorization does not exist.

\section{Unit triangular factors}\label{sec:unit_triangular}

We now discuss results where we require one of the factors to be unit triangular. Classical results consider the factorization with partial pivoting $PA = LU$ (\Cref{thm:row_permutation}). The row permutation $P$ allows us to obtain a unit lower triangular factor $L$. When a zero pivot is encountered, we perform a row pivoting to get a non zero pivot. If the entire column is equal to zero, we set the diagonal entry of $L$ to 1, the entries below to 0, and continue the factorization with the next column.

We are interested in conditions for the existence of $A=LU$ where $L$ is unit lower triangular. As before, this requires internal permutations, but this time they can only be applied to the columns. The triangular unit structure of $L$ prevents row permutations. The condition on $A$ is more stringent than in the general LU case and we require
\[
	\nl(A[1:k,1:k]) = \nl(A[:,1:k])
\]
for all $k$.

We note that, for a general matrix $A$, column pivoting in general is not sufficient to get a unit lower triangular factor. For example, consider the matrix
\[
	A = \begin{pmatrix}
		0 & 0 \\
		1 & 1
	\end{pmatrix}
\]
No column pivoting works. However, row pivoting is sufficient to get a unit lower triangular factor.

\Cref{tab:constraints} summarizes the constraints on the factorization, the condition (for all $k$), and the stringency of the condition.

\begin{table}[htbp]
	\centering
	\begin{tabular}{lp{5cm}p{7cm}}
		\toprule
		Constraint & Condition (for all $k$) & Stringency \\
		\midrule
		General LU & \raggedright $\nl(A_{11}) \le \nl(\text{Col}) + \nl(\text{Row})$ & Least stringent (allows mixed dependencies) \\
		Unit Lower L & $\nl(A_{11}) = \nl(\text{Col})$ & Moderate (requires column dependencies) \\
		Unit Upper U & $\nl(A_{11}) = \nl(\text{Row})$ & Moderate (requires row dependencies) \\
		Unit L and Unit U & \raggedright Only possible if $A$ is strongly non-singular & Most stringent (impossible for singular $A$) \\
		\bottomrule
	\end{tabular}
	\vspace{5pt}
	\caption{Constraints on factorization, condition (for all $k$), and stringency of condition.}\label{tab:constraints}
\end{table}

The theorem below should be compared with \Cref{thm:row_permutation}.

\begin{theorem}\label{thm:main_unit_lower}
	Let $A$ be a matrix of size $n$ and rank $r \le n$. $A$ has an LU factorization $A = LU$ where $L$ is square unit lower triangular if and only if for all $k = 1, \ldots, n$:
	\begin{equation}
		\label{eq:cond_unit_lower}
		\nl(A[1:k,1:k]) = \nl(A[:,1:k])
	\end{equation}

	\noindent\textbf{Construction:} If $A$ satisfies \Cref{eq:cond_unit_lower}, there exists a column permutation $P_c$ such that the matrix $B = A P_c^{-1}$ admits a factorization $B = L_B U_B$ where $L_B$ is unit lower triangular. Consequently,
	\[ A = L_B (U_B P_c) = L U \]
	is a valid factorization of $A$ with $L=L_B$ unit lower triangular and $U$ upper triangular.

	\noindent\textbf{Structure:} For the matrix $B$ constructed above, row $j$ is a linear combination of the previous rows if and only if column $j$ is a linear combination of the previous columns.

	\noindent\textbf{Sparsity:} Let $j_0$ be the column index in $A$ corresponding to the logical column index $j$ in $B$ (i.e., $e_{j_0} = P_c^T e_j$). The factors exhibit the following sparsity patterns:
	\begin{itemize}
		\item \textbf{Dependent Case:} If row $j$ of $B$ is a linear combination of the previous rows, then $j_0 \le j$. Furthermore:
		\[ L[:,j] = e_j, \quad U[j,:] = 0, \quad \text{and} \quad U[j_0:, j_0] = 0 \]
		\item \textbf{Independent Case:} If row $j$ of $B$ is linearly independent of the previous rows, then $j_0 \ge j$. Furthermore:
		\[ U[j+1:, j_0] = 0 \]
	\end{itemize}
\end{theorem}

\begin{proof} \par
	\textbf{Condition is necessary.} \quad Short proof: $U$ trivially satisfies \Cref{eq:cond_unit_lower} since it is upper triangular. Since $L[1:k,1:k]$ is full rank, $A$ also satisfies \Cref{eq:cond_unit_lower}.

	Direct proof: Assume that the matrix $A$ admits a factorization $A = LU$, where $L$ is a unit lower triangular matrix and $U$ is an upper triangular matrix.

	We analyze the structure of the matrices for an arbitrary $k \in \{1, \dots, n\}$.
	Because $U$ is upper triangular, the first $k$ columns of $A$ depend only on the first $k$ columns of $U$. Specifically, we can write the submatrices of $A$ as:

	\begin{itemize}
		\item Leading principal submatrix
		\[ A[1:k, 1:k] = L[1:k, 1:k] \; U[1:k, 1:k] \]
		\item Leading column block
		\[ A[:, 1:k] = L[:, 1:k] \; U[1:k, 1:k] \]
	\end{itemize}

	Since $L$ is \textbf{unit} lower triangular, its diagonal entries are all 1. This implies that any leading principal submatrix of $L$ is invertible.
	\begin{itemize}
		\item Let $L_{kk} = L[1:k, 1:k]$. Because $\det(L_{kk}) = 1$, $L_{kk}$ is non-singular (full rank).
		\item Let $L_{col} = L[:, 1:k]$. The top block of $L_{col}$ is $L_{kk}$. Since $L_{kk}$ is non-singular, the larger matrix $L_{col}$ must have full column rank (rank $k$).
	\end{itemize}

	We now compute the ranks of the submatrices of $A$ using the properties of $L$:
	\begin{itemize}
		\item For the principal submatrix: since $L_{kk}$ is non-singular, multiplying by it does not change the rank:
		\[ \rk(A[1:k, 1:k]) = \rk(L_{kk} \cdot U[1:k, 1:k]) = \rk(U[1:k, 1:k]) \]
		\item For the column block: since $L_{col}$ has full column rank, multiplying by it on the left preserves the rank of the matrix on the right:
		\[ \rk(A[:, 1:k]) = \rk(L_{col} \cdot U[1:k, 1:k]) = \rk(U[1:k, 1:k]) \]
	\end{itemize}

	Comparing the results, we see that the ranks are identical:
	\[ \rk(A[1:k, 1:k]) = \rk(A[:, 1:k]) \]

	By the Rank-Nullity Theorem, for any matrix $M$ with $k$ columns, $\nl(M) = k - \rk(M)$. Since both $A[1:k, 1:k]$ and $A[:, 1:k]$ have exactly $k$ columns:
	\begin{align*}
		\nl(A[1:k, 1:k]) & = k - \rk(A[1:k, 1:k]) \\
		\nl(A[:, 1:k])   & = k - \rk(A[:, 1:k])
	\end{align*}

	Therefore:
	\[ \nl(A[1:k, 1:k]) = \nl(A[:, 1:k]) \]
	This completes the proof that the condition is necessary.

	\noindent\textbf{Condition is sufficient.} \quad We follow a proof similar to the one for \Cref{thm:main2}. We do an induction on the size of $A$. The result is true for $n=1$.

	Assume it is true for size $n-1$. If $a_{11} \neq 0$, we verify that the result is true by the induction hypothesis.

	Assume $a_{11} = 0$. From \Cref{eq:cond_unit_lower}, we have $A[:,1] = 0$.

	Case 1: If row 1 of $A$ is non-zero, we can use the same steps as in the proof as \Cref{thm:main2} with a column permutation to prove the result by induction. In addition, this shows that if $e_{j_0} = P_c^T e_j$ and if row $j$ of $B$ is linearly independent of the previous rows then $j_0 \ge j$ and $U[j+1:,j_0] = 0$.

	Case 2: If row 1 of $A$ is zero, we cannot use the same proof as \Cref{thm:main2} because we can no longer permute rows. However, since row 1 of $A$ is zero, we can define:
	\[
		A = IDU
	\]
	with
	\[
		D = \begin{pmatrix}
			1 & 0 \\
			0 & S
		\end{pmatrix}, \quad
		U = \begin{pmatrix}
			0 & A[1,2:n] \\
			0 & I
		\end{pmatrix} = \begin{pmatrix}
			0 & 0 \\
			0 & I
		\end{pmatrix}
	\]
	and $S=A_{22}$. $D$ satisfies
	\[
		\nl(D[1:k,1:k]) = \nl(A[1:k,1:k]) - 1
		= \nl(A[:,1:k]) - 1 = \nl(D[:,1:k])
	\]
	$D$ satisfies \Cref{eq:cond_unit_lower}. By the induction hypothesis, we can find a column permutation $P_{c,S}$ such that $S P_{c,S}^{-1}$ has an LU factorization with unit lower triangular factor.

	This also proves that row $j$ of $B$ is a linear combination of the previous rows if and only if column $j$ of $B$ is a linear combination of the previous columns. Moreover, if $e_{j_0} = P_c^T e_j$ and row $j$ of $B$ is a linear combination of the previous rows, then $j_0 \le j$. We have $L[:,j] = e_j$, $U[j,:] = 0$, and $U[j_0:,j_0] = 0$.

	Note that the step above with $D$ is not a valid step in \Cref{thm:main2}. Indeed:
	\begin{align*}
		\nl(D[1:k,1:k])   & = \nl(A[1:k,1:k]) - 1 \\
		\nl(D[:,1:k])     & = \nl(A[:,1:k]) - 1 \\
		\nl({D[1:k,:]}^T) & = \nl({A[1:k,:]}^T) - 1
	\end{align*}
	So that \Cref{prop:rank} may not be satisfied by $D$.
\end{proof}

We now state the corresponding result where $U$ is unit upper triangular. Please compare the result below with \Cref{cor:col_permutation}.

\begin{corollary}\label{corr:main_unit_upper}
	Let $A$ be a matrix of size $n$ and rank $r \le n$. $A$ has an LU factorization $A = LU$ where $U$ is unit upper triangular if and only if for all $k = 1, \ldots, n$:
	\begin{equation}
		\label{eq:cond_unit_upper}
		\nl(A[1:k,1:k]) = \nl({A[1:k,:]}^T)
	\end{equation}

	\noindent\textbf{Construction:} If $A$ satisfies \Cref{eq:cond_unit_upper}, there exists a row permutation $P_r$ such that the matrix $B = P_r^{-1} A$ admits a factorization $B = L_B U_B$ where $U_B$ is unit upper triangular. Consequently,
	\[ A = (P_r L_B) U_B = L U \]
	is a valid factorization of $A$ with $U=U_B$ unit upper triangular and $L$ lower triangular.

	\noindent\textbf{Structure:} For the matrix $B$ constructed above, column $i$ is a linear combination of the previous columns if and only if row $i$ is a linear combination of the previous rows.

	\noindent\textbf{Sparsity:} Let $i_0$ be the row index in $A$ corresponding to the logical row index $i$ in $B$ (i.e., $e_{i_0} = P_r e_i$). The factors exhibit the following sparsity patterns:
	\begin{itemize}
		\item \textbf{Dependent Case:} If column $i$ of $B$ is a linear combination of the previous columns, then $i_0 \le i$. Furthermore:
		\[ U[i,:] = e_i, \quad L[:,i] = 0, \quad \text{and} \quad L[i_0, i_0:] = 0 \]
		\item \textbf{Independent Case:} If column $i$ of $B$ is linearly independent of the previous columns, then $i_0 \ge i$. Furthermore:
		\[ L[i_0, i+1:] = 0 \]
	\end{itemize}
\end{corollary}

\begin{proof}
	Consider $B = A^T$. Apply \Cref{thm:main_unit_lower}.
\end{proof}

\subsection{Python implementation}

In \Cref{sec:python_unit_lower}, we provide a Python implementation of the algorithm. This code is significantly simpler since we only do internal permutations on the columns of $A$. As previously, the internal permutations are done implicitly, and this algorithm is not backward stable.

\section{Conclusion}

In this paper, we have provided an analysis of the necessary and sufficient conditions for the existence of an LU factorization for arbitrary square matrices, without the use of pivoting. We established that the existence of such a factorization is determined by a specific inequality relating the nullity of the leading principal submatrices to the nullity of the corresponding leading column and row blocks. This condition generalizes the classical requirement for non-singular matrices (non-vanishing leading principal minors) to the general rank-deficient case.

Our approach utilizes a constructive proof based on induction. This method not only simplifies the theoretical derivation but also provides explicit insight into how singular blocks can be handled through internal permutations that preserve the overall triangular structure of the factors.

Furthermore, we defined and characterized rank-revealing LU factorizations without global pivoting. We derived tight sparsity bounds for the resulting factors $L$ and $U$, linking the physical indices of the non-zero entries to the logical rank profile of the matrix. Finally, we examined the constrained cases where one factor is required to be unit triangular, identifying the specific rank conditions required and their duality with respect to row and column permutations. These results offer a unified framework for understanding matrix factorizations in the presence of singularity and rank deficiency.

\appendix

\section{Appendix}\label{sec:appendix}

We prove a few additional results that are helpful in the main proofs. For additional background on connections between the LU factorization, Schur complements, and oblique projections, please refer to~\cite{Ouellette_1981,Zhang_2006,Householder_2013}.

\begin{lemma}\label{lem:oblique_projection}
	Denote by $X$ and $Y$ two matrices of size $n \by k$. Assume that $X$ is full rank and that $Y^T X$ is non-singular. Define:
	\[ P = X {(Y^T X)}^{-1} Y^T \]
	Then $P$ is the oblique projection onto the range of $X$ along the orthogonal complement of the range of $Y$.
\end{lemma}

\begin{proof}
	We have:
	\[ P X = X {(Y^T X)}^{-1} Y^T X = X \]
	So the range of $P$ contains the range of $X$. Since $\rk(P) \le k$ and $\rk(X)=k$, the range of $P$ is exactly the range of $X$.

	Consider $v$ in the orthogonal complement of the range of $Y$ (i.e., $Y^T v = 0$). We have:
	\[ P v = X {(Y^T X)}^{-1} Y^T v = 0 \]
	Thus, the null space of $P$ contains the orthogonal complement of the range of $Y$. By the rank-nullity theorem, the dimensions match, so the null space of $P$ is exactly the orthogonal complement of the range of $Y$.
\end{proof}

\begin{lemma}\label{lem:oblique_projection_specific}
	Denote by $X = [x_1, \ldots, x_k]$ a rank $k$ matrix. Consider the oblique projection $P_o(X)$ onto
	\[ S = \spn\{ x_1,\dots,x_k \} \] 
	along the subspace of vectors with the first $k$ entries equal to zero. Assume that $X[1:k,:]$ is non-singular. Then:
	\[ P_o(X) = X {(X[1:k,:])}^{-1} Y^T \]
	where
	\[ Y = \begin{pmatrix}
			I_k \\ 0
		\end{pmatrix} \]
\end{lemma}

\begin{proof}
	In \Cref{lem:oblique_projection}, set $Y$ as
	\[ Y = \begin{pmatrix}
			I_k \\ 0
		\end{pmatrix} \]
	The range of $Y$ is the span of the first $k$ standard basis vectors. Its orthogonal complement is exactly the subspace of vectors with the first $k$ entries equal to zero.
\end{proof}

Denote by $Q_o(X) = I - P_o(X)$ the oblique projection onto the subspace of vectors with the first $k$ entries equal to zero along $S=\spn(X)$.

Consider the state of the matrix at the beginning of step $k$ of the LU factorization algorithm (after $k-1$ columns have been eliminated). The Schur complement can be interpreted in terms of the oblique projection $Q_o$. Specifically, from \Cref{lem:oblique_projection_specific} applied to the first $k-1$ columns, we have:
\[ S^o_k = A[:,k:] - A[:,1:k-1] {A[1:k-1,1:k-1]}^{-1} A[1:k-1,k:]
	= Q_o(A[:,1:k-1]) \; A[:,k:]
\]
The standard Schur complement $S_k$ corresponds to the bottom block of this projected matrix: $S_k = S^o_k[k:,:]$.

\begin{lemma}\label{lem:oblique}
	If column $j$ of $A$ (where $j \ge k$) is linearly dependent on columns $1$ to $k-1$ of $A$, then column $j - k + 1$ of $S^o_k$ is zero.
\end{lemma}

\begin{proof}
	Since column $j$ of $A$ is linearly dependent on columns $1$ to $k-1$, we have:
	\[ A[:,j] \in \spn\{ A[:,1], \ldots, A[:,k-1] \} \] 
	Let $X = A[:,1:k-1]$. Then $A[:,j] \in \spn(X)$. By definition of the projection $P_o(X)$:
	\[ P_o(X) A[:,j] = A[:,j] \]
	Therefore:
	\[ Q_o(X) A[:,j] = (I - P_o(X)) A[:,j] = 0 \]

	Since $S^o_k$ is formed by applying $Q_o(X)$ to the columns $A[:,k:]$, the column corresponding to $A[:,j]$ is zero.
\end{proof}

\subsection{Python code for LU factorization}\label{sec:python_restricted_pivoting}

See \Cref{lst:lu_restricted_pivoting}.

\begin{lstlisting}[style=PythonStyle,
	caption={Algorithm for LU factorization. Returns the LU factorization of a matrix $A = LU$ or None if the factorization does not exist. The algorithm is guaranteed to return None only when the factorization does not exist.},
	label={lst:lu_restricted_pivoting},
	% Tell listings that markers are preceded by "# "
	rangeprefix=\#\ ,
	% Define the range to print
	linerange=scifig_start-scifig_end,
	% Do not include the marker lines themselves in the output
	includerangemarker=false]
import numpy as np

# scifig_start
def lu_factorization(A_in):
    A = A_in.astype(np.float64).copy()
    n = A.shape[0]
    L, U = np.zeros_like(A), np.zeros_like(A)
    for k in range(n):
        i, j = k, k
        if np.isclose(A[i, j], 0.0):        
            # Search for the first shell `i` (relative to k) that contains a non-zero element in either its row or column slice.
            is_nz = ~np.isclose(A[k:, k:], 0.0)
            
            # valid_rows[x] is True if Row x of A[k:, k:] (upper part, j>=x) has non-zero
            valid_rows = np.any(np.triu(is_nz), axis=1)
            # valid_cols[x] is True if Col x of A[k:, k:] (lower part, i>=x) has non-zero
            valid_cols = np.any(np.tril(is_nz), axis=0)    
            valid_shells = valid_rows | valid_cols     

            if np.any(valid_shells):
                l_rel = np.argmax(valid_shells) # First index where valid
                l = k + l_rel                
                # Find k_row and k_col for this specific shell l                
                k_row = l + np.argmax(~np.isclose(A[l, l:], 0.0)) if valid_rows[l_rel] else n + 1
                k_col = l + np.argmax(~np.isclose(A[l:, l], 0.0)) if valid_cols[l_rel] else n + 1                
                use_row = k_row <= k_col
                if l == k and np.any(~np.isclose(A[k:, k] if use_row else A[k, k:], 0.0)):
                    # Factorization does not exist
                    return None         
                i, j = (l, k_row) if use_row else (k_col, l)
            else:
                break

        L[k:, k] = A[k:, j] / A[i, j] #(*@\label{copy_leftward}@*)
        U[k, k:] = A[i, k:]
        A[k+1:, k+1:] -= np.outer(L[k+1:, k], U[k, k+1:]) #(*@\label{zero_column}@*)
    
    return L, U
# scifig_end

if __name__ == "__main__":

    # Test suite
    np.random.seed(42)
    
    # 1. Random Matrix
    # Product of a random lower and upper triangular 
    # matrix with entries in [0,1] should pass
    print("\nTest 1 (Random):")
    L = np.tril(np.random.rand(10, 10))
    U = np.triu(np.random.rand(10, 10))
    A1 = L @ U
    res1 = lu_factorization(A1)
    if res1: 
        print(f"  A=LU? {np.allclose(A1, res1[0] @ res1[1])}")
        print(f"  L Lower? {np.allclose(res1[0], np.tril(res1[0]))}")
        print(f"  U Upper? {np.allclose(res1[1], np.triu(res1[1]))}")
    else:
        print("  FAIL.")        

    # 2. Error Case (Row!=0 Col!=0)
    A2 = np.array([[0., 1.], [1., 1.]])
    # Code should ERROR per logic: "if not row_zero and not col_zero: Error"
    print("\nTest 2 (Error Case)")
    res2 = lu_factorization(A2)
    if res2 is None: 
        print("PASS: Correctly returned None.")
    else:
        print(A2)
        print(res2[0])
        print(res2[1])
        print(res2[0]@res2[1])
        print("  FAIL.")

    # 3. Pivot Needed (Row=0, Col!=0) -> Swap Row
    A3 = np.array([[0,0,0], [1,2,3], [0,5,6]], dtype=float)
    print(f"\nTest 3 (Row Swap):")
    res3 = lu_factorization(A3)
    if res3: 
        print(f"  A=LU? {np.allclose(A3, res3[0] @ res3[1])}")
    else:
        print("  FAIL.")

    # 4. Pivot Needed (Row!=0, Col=0) -> Swap Col
    A4 = np.array([[0,5,2], [0,1,3], [0,4,8]], dtype=float)
    print(f"\nTest 4 (Col Swap):")
    res4 = lu_factorization(A4)
    if res4:
        print(f"  A=LU? {np.allclose(A4, res4[0] @ res4[1])}")
    else:
        print("  FAIL.")    
    
    # 5. Singular Block (Both 0) -> Search Submatrix
    # Need matrix where A[0,0]=0, row0=0, col0=0, but A[1:,1:] not zero
    # AND rank > 0
    A5 = np.array([[0,0,0,0], [0,0,0,0], [0,0,1,2], [0,0,3,4]], dtype=float)
    # Pivot 0: Both 0. Submatrix A[0:,0:] has non-zeros.
    # Logic: Search A[1:, 0:] -> Rows 2 and 3 are non-zero.
    # Swap row 0 with row 2 (index 2).
    # New A: Row 0 is [0,0,1,2], Row 2 is [0,0,0,0].
    # Now Row 0 non-zero. Find col j. Col 2 is 1. Swap Col 0 and 2.
    # New A: A[0,0]=1. Proceed.
    print(f"\nTest 5 (Singular Block):")
    res5 = lu_factorization(A5)
    if res5:
        print(f"  A=LU? {np.allclose(A5, res5[0] @ res5[1])}")
        print(f"  L Lower? {np.allclose(res5[0], np.tril(res5[0]))}")
        print(f"  U Upper? {np.allclose(res5[1], np.triu(res5[1]))}")
    else:
        print("  FAIL.")

    # Randomized Large Scale
    print(f"\nRunning Randomized Tests")
    n_tests = 100
    p = 0.1
    n = 100

    # Failing tests
    success = 0    
    for _ in range(n_tests):
        r = np.random.randint(1, n-1)
        L_gen = np.tril((np.random.rand(n,r)<p).astype(float)) + np.eye(n,r)
        U_gen = np.triu((np.random.rand(r,n)<p).astype(float)) + np.eye(r,n)
        A_rand = L_gen @ U_gen
        # Create a random matrix of size n-r
        A_botright = 1 + np.random.rand(n-r, n-r)
        A_botright[0,0] = 0
        # Add to the bottom right part of A_rand
        A_rand[r:, r:] += A_botright
        res = lu_factorization(A_rand)
        # This should fail
        if res is None: success += 1
        else: print("FAIL.")

    print(f"Impossible Cases: {success}/{n_tests} tests passed.")  

    success = 0
    import time
    start_time = time.time()
    for _ in range(n_tests):
        L_gen = np.tril((np.random.rand(n, n) < p).astype(float))
        U_gen = np.triu((np.random.rand(n, n) < p).astype(float))
        A_rand = L_gen @ U_gen
        res = lu_factorization(A_rand)
        if (res and np.allclose(A_rand, res[0] @ res[1]) 
            and np.allclose(res[0], np.tril(res[0])) 
            and np.allclose(res[1], np.triu(res[1]))): success += 1
        else:
            # Print diagnostics
            print("FAIL.")
            print(A_rand)
    end_time = time.time()
    print(f"Successful Cases: {success}/{n_tests} tests passed.")
    print(f"  Time taken for {n_tests} tests: {end_time - start_time:.4f} seconds")    
\end{lstlisting}

\subsection{Python code for classical LU with full pivoting}\label{sec:python_full_pivoting}

See \Cref{lst:lu_full_pivoting}.

\begin{lstlisting}[style=PythonStyle, caption={Classical LU factorization with full pivoting. Returns a factorization of the form $PAQ = LU$.},
	label={lst:lu_full_pivoting},
	% Tell listings that markers are preceded by "# "
	rangeprefix=\#\ ,
	% Define the range to print
	linerange=scifig_start-scifig_end,
	% Do not include the marker lines themselves in the output
	includerangemarker=false]
import numpy as np

# scifig_start
def lu_pivoting(A_in):
    A = A_in.astype(np.float64).copy()
    n = A.shape[0]
    L, U = np.zeros_like(A), np.zeros_like(A)
    p, q = np.arange(n), np.arange(n)

    def swap_rows(k, i):
        A[[k, i], k:] = A[[i, k], k:]
        p[[k, i]] = p[[i, k]]
        L[[k, i], :k] = L[[i, k], :k]

    def swap_cols(k, j):
        A[k:, [k, j]] = A[k:, [j, k]]
        q[[k, j]] = q[[j, k]]
        U[:k, [k, j]] = U[:k, [j, k]]

    for k in range(n):
        idx_max = np.argmax(np.abs(A[k:, k:]))
        i_local, j_local = np.unravel_index(idx_max, A[k:, k:].shape)
        i_max, j_max = k + i_local, k + j_local
        swap_rows(k, i_max)
        swap_cols(k, j_max)
        if np.isclose(A[k, k], 0.0): break
        
        L[k:, k] = A[k:, k] / A[k, k]
        U[k, k:] = A[k, k:]
        A[k+1:, k+1:] -= np.outer(L[k+1:, k], U[k, k+1:])

    P, Q = np.eye(n)[p], np.eye(n)[:, q]
    return P, Q, L, U
# scifig_end

if __name__ == "__main__":
    np.random.seed(42)
    n_tests = 100; p = 0.5; n = 100; success = 0
    import time
    start_time = time.time()    
    for _ in range(n_tests):
        A_rand = np.random.rand(n,n)
        res = lu_pivoting(A_rand)
        if (res and np.allclose(res[0] @ A_rand @ res[1], res[2] @ res[3]) 
            and np.allclose(res[2], np.tril(res[2])) 
            and np.allclose(res[3], np.triu(res[3]))): success += 1
    end_time = time.time()
    print(f"  Time taken for {n_tests} tests: {end_time - start_time:.4f} seconds")            
    print(f"PASS: {success}/{n_tests} tests passed.")
\end{lstlisting}

\subsection{Python code for unit lower triangular LU factorization}\label{sec:python_unit_lower}

See \Cref{lst:unit_lower_lu}.

\begin{lstlisting}[style=PythonStyle,
	caption={Unit lower triangular LU factorization. Returns a factorization $A = LU$ where $L$ is unit lower triangular or None if the factorization does not exist.},
	label={lst:unit_lower_lu},
	% Tell listings that markers are preceded by "# "
	rangeprefix=\#\ ,
	% Define the range to print
	linerange=scifig_start-scifig_end,
	% Do not include the marker lines themselves in the output
	includerangemarker=false]
import numpy as np

# scifig_start
def unit_lower_lu_factorization(A_in):
    A = A_in.astype(np.float64).copy()
    n = A.shape[0]
    L, U = np.eye(n), np.zeros_like(A)

    for k in range(n):
        j = k + np.argmax(~np.isclose(A[k, k:], 0.0))
        if np.isclose(A[k, k], 0.0):
            if not np.allclose(A[k+1:, k], 0.0):
                # Factorization does not exist
                return None
            if np.isclose(A[k, j], 0.0): continue

        L[k+1:, k] = A[k+1:, j] / A[k, j]
        U[k, k:] = A[k, k:]
        A[k+1:, k+1:] -= np.outer(L[k+1:, k], U[k, k+1:])
    
    return L, U
# scifig_end

if __name__ == "__main__":
    np.random.seed(42)
    # 1. Random Matrix
    A1 = np.random.rand(10, 10)
    res1 = unit_lower_lu_factorization(A1)
    if res1:
        print(f"Test 1 (Random):\n  A=LU? {np.allclose(A1, res1[0] @ res1[1])}")
        print(f"  Unit Lower? {np.allclose(res1[0], np.tril(res1[0])) and np.allclose(np.diag(res1[0]), 1)}")
        print(f"  U Upper? {np.allclose(res1[1], np.triu(res1[1]))}")
    else:
        print("Test 1 (Random): Failed")

    # 2. Impossible Case (Pivot=0, Col!=0)
    A2 = np.array([[0,0], [1,0]])
    print("\nTest 2 (Impossible)")
    if unit_lower_lu_factorization(A2) is None: print("PASS: Correctly returned None.")

    # 3. Pivot Needed (Pivot=0, Row!=0, Col=0) -> Swap Col
    A3 = np.array([[0, 2], [0, 4]]) 
    res3 = unit_lower_lu_factorization(A3)
    if res3:
        print(f"\nTest 3 (Pivot): A=LU? {np.allclose(A3, res3[0] @ res3[1])}")
    else:
        print("Test 3 (Pivot): Failed")
    
    # 4. Singular Block (Pivot=0, Row=0, Col=0)
    A4 = np.array([[0,0], [0,1]])
    res4 = unit_lower_lu_factorization(A4)
    if res4:
         print(f"Test 4 (Singular): A=LU? {np.allclose(A4, res4[0] @ res4[1])}")
    else: 
        print("Test 4 (Singular): Failed")

    # Randomized Large Scale Tests
    print(f"\nRunning Randomized Tests")
    n_tests = 100
    p = 0.1
    n = 100

    success = 0
    for _ in range(n_tests):
        L_gen = np.tril(np.random.rand(n,n)<p, -1) + np.eye(n)
        U_gen = np.triu(np.random.rand(n,n)<p)        
        # Choose a random column between 1 and n-2
        j = np.random.randint(0, n-1)
        L_gen[j, j] = 0
        L_gen[j+1:, j] = 1 + np.random.rand(n-j-1)
        U_gen[:j+1, :j+1] = np.eye(j+1) + np.triu(np.random.rand(j+1,j+1), 1)
        A_rand = L_gen @ U_gen
        res = unit_lower_lu_factorization(A_rand)
        if (res is None): success += 1            
    print(f"Impossible Cases: {success}/{n_tests} tests passed.")    

    success = 0    
    for _ in range(n_tests):
        L_gen = np.tril(np.random.rand(n,n)<p, -1) + np.eye(n)
        U_gen = np.triu(np.random.rand(n,n)<p)
        A_rand = L_gen @ U_gen
        res = unit_lower_lu_factorization(A_rand)
        if (res 
            and np.allclose(A_rand, res[0] @ res[1])
            and np.allclose(res[0], np.tril(res[0])) 
            and np.allclose(np.diag(res[0]), 1.0)
            and np.allclose(res[1], np.triu(res[1]))): success += 1
    print(f"Successful Cases: {success}/{n_tests} tests passed.")
\end{lstlisting}

\clearpage
\phantomsection 
\bibliographystyle{unsrtnat}
\bibliography{ref}

\end{document}